\documentclass[10pt,a4paper]{amsart}  

\usepackage{bbm,amsmath,amssymb,amsfonts,graphicx,color,subfig,booktabs,multirow,diagbox,mathtools}
\usepackage[hidelinks]{hyperref}

\usepackage[english]{babel}

\newcommand{\R}{\mathbb{R}}

\newcommand{\Ccal}{\mathcal{C}}

\def\spheresdp/{A}
\def\spacesdp/{B}

\newtheorem{defin}{Definition}[section]

\newtheorem{proposition}[defin]{Proposition}

\newtheorem{theorem}[defin]{Theorem}
\newtheorem{corollary}[defin]{Corollary}
\newtheorem{lemma}[defin]{Lemma}

\theoremstyle{definition}

\begin{document}

\title[Exact semidefinite programming bounds for packing problems]{Exact semidefinite programming bounds\\ for packing problems}

\author{Maria Dostert}
\address{M.~Dostert, \'{E}cole Polytechnique F\'{e}d\'{e}rale de Lausanne, SB TN, Station 8, 1015 Lausanne, Suisse / Switzerland}
\email{maria.dostert@epfl.ch}

\author{David de Laat}
\address{D.~de Laat, Delft Institute of Applied Mathematics, Delft
  University of Technology, P.O. Box 5031, 2600 GA Delft, The
  Netherlands} 
\email{d.delaat@tudelft.nl}

\author{Philippe Moustrou}
\address{P.~Moustrou, Department of Mathematics and Statistics, UiT - the Arctic University of Norway,
9037 Tromsø, Norway}
\email{philippe.moustrou@uit.no}



\date{July 9, 2020} 

\begin{abstract}
In this paper we give an algorithm to round the floating point output of a semidefinite programming solver to a solution over the rationals or a quadratic extension of the rationals.
This algorithm does not require the solution to be strictly feasible and works for large problems. 
We apply this to get sharp bounds for packing problems, and we use these sharp bounds to prove that certain optimal packing configurations are unique up to rotations. In particular, we show that the configuration coming from the $\mathsf{E}_8$ root lattice is the unique optimal code with minimal angular distance $\pi/3$ on the hemisphere in $\R^8$, and we prove that the three-point bound for the $(3, 8, \vartheta)$-spherical code, where $\vartheta$ is such that $\cos \vartheta = (2\sqrt{2}-1)/7$, is sharp by rounding to $\mathbb Q[\sqrt{2}]$. We also use our machinery to compute sharp upper bounds on the number of spheres that can be packed into a larger sphere.
\end{abstract}

\maketitle

\section{Introduction}

In this paper we consider the problem of extracting an exact solution from the  output of a numerical semidefinite programming solver. A particularly fitting application domain for this is extremal geometry, including packing problems where we ask for the size of a largest independent set in a graph whose set of vertices is a compact space.
A typical example is the \emph{$\vartheta$-spherical code} problem, where the vertex set is the unit sphere, and two distinct vertices $x$ and $y$ are adjacent when the inner product between $x$ and $y$ is at most $ \cos \vartheta$. The strongest known upper bounds   often use semidefinite programming \cite{bachoc08, laat15,dostert17}, which is an extension of linear programming where one optimizes over positive semidefinite matrices satisfying linear constraints. 
Since semidefinite programs are solved in floating point arithmetic, turning the numerical bounds into rigorous upper bounds requires additional work. 

To prove that a packing configuration of size $N$ is optimal, we just need any upper bound in the interval $[N,N+1)$. 
In this case, extracting a rigorous bound is easy, because it is enough to prove the existence of an exact solution of the semidefinite program whose objective is close to the objective of the floating point solution, which means we simply need to round a strictly feasible solution. 
However, if we want to prove uniqueness of such a configuration (see Section~\ref{info-exact-solution}), or if we aim at solving more general problems in extremal geometry where the optimal value is not an integer (for instance, energy minimization problems \cite{Cohn12, laat19}), then we need an exact bound. 
Extremal geometry thus provides large semidefinite programming problems where we look for an exact optimal solution, that may lie on the boundary of the positive semidefinite matrix cone.	
%To get an exact \emph{sharp} bound is more difficult though, because  we have to find an exact solution in the optimal face of the semidefinite program.

For the problem of finding exact sum of squares decompositions of polynomials, there exist several hybrid numeric-symbolic algorithms \cite{MR2474341, MR4072646, MR2500392, MR2854844, magron2018exact}, and some of them \cite{MR2500392, MR2854844} can be applied for special instances of polynomials that are not strictly feasible.
For general semidefinite programs, as required in our context, there exist symbolic algorithms \cite{MR3574590, MR3707858, MR3840381}, but they cannot be used for the size of problems we consider. 
Other approaches rely on facial reduction: While \cite{MR4044453} provides an exact method for sum of squares decompositions, more general techniques such as \cite{MR3844533} remain purely numerical. 
We propose a hybrid numeric-symbolic algorithm, which can be related to facial reduction, and is able to tackle large semidefinite problems arising in extremal geometry, where the optimal solutions are not necessarily strictly feasible.

\medskip

Since we consider problems where we have a candidate optimal configuration for which we have a numerically sharp semidefinite programming bound, we can include the objective value as a linear constraint, and we are only looking for a solution of the corresponding feasibility problem. The semidefinite programming solver returns a near feasible solution of this semidefinite program, given as a list of matrices whose entries are floating point numbers. For this solution the linear constraints are almost satisfied, and the eigenvalues of the matrices are positive or close to zero (but not necessarily nonnegative). 
If we expect to find a feasible solution with entries in a given field, the challenge consists in turning the floating point values into elements of this field in such a way that the linear constraints are satisfied and the matrices are positive semidefinite. 

The solver gives an approximate solution in the relative interior of the feasible set. This means that if the feasible set has the same dimension as the affine space defined by the linear constraints, the near zero eigenvalues will become exactly zero after projecting the approximate solution into the affine space. Since any positive eigenvalue bounded away from zero will remain positive after a small perturbation of the matrix entries, we thus find an exact feasible solution simply by projecting the approximate solution into the affine space. In prior works, the approach has been to include  additional linear constraints coming from the complementary slackness conditions of an optimal configuration in the hope that the above condition becomes valid \cite{Cohn12, bachoc09b}. However, in general, there is no guarantee at all that this works, and we were not able to prove Theorem~\ref{thm:UniquenessHemisphere} with this approach, which was the original motivation for this project. 

In this paper we develop a general procedure to extract an exact optimal solution from the numerical optimal solution of a semidefinite program. Our main idea is the following: We want to understand the eigenvectors corresponding to near-zero eigenvalues, in order to force these eigenvalues to become exactly zero after rounding. This task seems to be challenging, since the kernels can be big, and the solver does not take into account their structure: Even if these linear spaces afford a basis over the algebraic field we want to round to, computing a basis in floating point arithmetic will only provide messy approximations of linear \emph{real} combinations of such vectors. In order to extract a suitable basis from these numerical approximations, we use the LLL algorithm to detect equations that have to be satisfied by the kernel vectors. This provides all additional linear constraints that will ensure the semidefiniteness of the matrices after rounding. In general, even if the constraints and the objective are rational, an optimal solution of a semidefinite program might require high algebraic degree \cite{MR2546336}. However, for every problem we considered, the semidefinite program and the conditions implied by an optimal configuration were defined over the same field (which was either the field of rational numbers or a quadratic field), and we were able to find an exact optimal solution over this field.  

\medskip
   
A fundamental approach to compute upper bounds on the size of spherical $\vartheta$-codes is the linear programming bound of Delsarte, Goethals, and Seidel \cite{delsarte77}. Even if the semidefinite programming \emph{three-point bound} by Bachoc and Vallentin \cite{bachoc08} provides stronger bounds, there are very few known cases where the semidefinite programming bound is sharp while the linear programming bound is not. One such example is given in \cite{bachoc09b}, where Bachoc and Vallentin prove optimality and uniqueness of the Petersen code, by using \emph{ad-hoc} techniques to show that their bound is sharp. Our rounding procedure directly turns a numerical solution of the semidefinite program into an exact rational solution, and we recover the uniqueness of the Petersen code; see Section~\ref{sec:sphericalCodes}.
 
In \cite{Schuette51}, Sch\"utte and van der Waerden prove optimality of the $\vartheta$-code in $S^2$ with cardinality $8$ and $\cos \vartheta = (2\sqrt{2}-1)/7$, and Danzer \cite{Danzer86} proves uniqueness up to rotations. The proofs of these results are purely geometric and quite technical. Still, in \cite{bachoc09b} Bachoc and Vallentin mention that their numerical computations suggest that the semidefinite programming bound is tight. However, they do not provide an exact optimal solution, and due to the value of $\cos \vartheta$ it seems that there is no optimal rational solution. With the adaptation of our approach to quadratic fields, we obtain an optimal solution over $\mathbb{Q}[\sqrt{2}]$. Based on our optimal solution, we give a simplified uniqueness proof in Section~\ref{sec:sphericalCodes}.

An even more challenging problem is to determine the optimal size of a spherical code in a spherical cap, where the spherical cap with center $e\in S^{n-1}$ and angle $\phi$ is defined by
\[
\mathrm{Cap}^{n-1}(e, \phi) = \big\{ x \in S^{n-1}: e \cdot x  \geq \cos (\phi)\big\}.
\]
In this situation, the linear programming bound cannot be applied, but in \cite{bachoc092} Bachoc and Vallentin adapt their three-point bound, that becomes a two-point bound in this context, to get semidefinite programming upper bounds. For the hemisphere $\mathrm{Cap}^{n-1}(e, \pi/2)$, they got a numerically sharp bound for $n=8$, which is closely related to the famous $\mathsf{E}_8$ lattice: The $240$ minimal vectors of $\mathsf{E_8}$ give the unique optimal spherical $\pi/3$-code in dimension $8$ \cite{kabatiansky78, sloane81, bannai04}. If $e$ is any of these minimal vectors, the intersection of this configuration with $\mathrm{Cap}^7(e, \pi/2)$ is a $\pi/3$-code of cardinality $183$. Bachoc and Vallentin get a numerical bound very close to $183$, which proves the optimality of this configuration. Moreover they conjecture that this is the only optimal configuration up to isometry. Here by using our machinery, we provide an exact optimal rational solution, and prove this conjecture; see Section~\ref{sec:Hemisphere}.  

To find more sharp bounds, we also apply our techniques to the similar problem of packing unit spheres in a larger sphere, which has connections to material science, radio-surgical treatments, and communication theory; see \cite{Hallah2013}. We use our machinery of rounding to rationals and quadratic fields to find several exact sharp bounds, and we also use this to give families of sharp bounds for all dimensions; see Section~\ref{sec:ballsinballs}.

\section{Semidefinite programming bounds for packing problems}\label{sec:2sdpbound}

\subsection{Characterizations of invariant kernels}

In Section~\ref{sec:2and3} we give a derivation of the semidefinite programming bounds we use in this paper. For this we need to characterize certain invariant positive definite kernels. 

We start with the well-known case of $O(n)$-invariant kernels on the sphere. Let $K$ be a positive definite kernel on $S^{n-1}$, by which we mean it is a continuous function from $S^{n-1} \times S^{n-1}$ to $\R$ for which the matrices 
\[
\begin{pmatrix}K(x_i,x_j)\end{pmatrix}_{i,j=1}^m
\]
are positive semidefinite for all $m \ge 0$ and $x_1,\ldots,x_m \in S^{n-1}$. Such a kernel is said to be $O(n)$-invariant if 
$
K(\gamma x, \gamma y) = K(x,y)
$
for all $\gamma \in O(n)$ and $x,y \in S^{n-1}$. 
For each $k \geq 0$, let $P_k^n$ be the degree $k$ ultraspherical polynomial for $S^{n-1}$ normalized such that $P_k^n(1) = 1$. These are also known as the Gegenbauer polynomials with parameter $n/2-1$. The functions $(x,y) \mapsto P_k^n(x \cdot y)$ are $O(n)$-invariant positive definite kernels on $S^{n-1}$, and Schoenberg's characterization  says that each positive definite $O(n)$-invariant kernel on $S^{n-1}$ is of the form 
\[
K(x,y) = \sum_{k=0}^\infty c_k P_k^n(x \cdot y),
\]
for nonnegative numbers $c_k \geq 0$, where convergence is uniform and absolute \cite{schoenberg42}. 

Bachoc and Vallentin give a characterization for the $O(n-1)$-invariant kernels on the sphere \cite{bachoc08}. As they observe in \cite{bachoc092} this also provides a characterization for the $O(n-1)$-invariant kernels on a spherical cap 
\[
\mathrm{Cap}^{n-1}(e,\varphi) = \big\{ x \in S^{n-1} : x \cdot e \geq \cos(\varphi) \big\},
\]
where we view $O(n-1)$ as the stabilizer subgroup of $O(n)$ with respect to some fixed point $e \in S^{n-1}$.

To state the proposition we define the matrix $Y_k^n(u,v,t)$ by 
\[
Y_k^n(u,v,t)_{i,i'} = u^i v^{i'} ((1-u^2)(1-v^2))^{k/2}P_k^{n-1}\left(\frac{t-uv}{\sqrt{(1-u^2)(1-v^2)}}\right)
\]  
and its symmetrization
\[
\overline{Y_k^n}(u,v,t) = \frac{Y_k^n(u,v,t) + Y_k^n(v,u,t)}{2}.
\]
Given a topological space $X$ we use the notation $\Ccal(X \times X)$
for the space of continuous functions $X \times X \to \R$ and $\Ccal(X \times X)_{\succeq 0}$ for the cone of positive definite kernels on $X$. Furthermore, we denote by $\langle A, B \rangle$ the trace inner product $\mathrm{Trace}(AB^{\sf T})$.

The following proposition is by Bachoc and Vallentin \cite{bachoc08}, where the last part  about the uniform limit follows immediately from their work in combination with Theorem A.8 from \cite{laat19}.
\begin{proposition}\label{prop:Bachoc-Vallentin}
Let $e \in S^{n-1}$. For each integer $d \ge 0$ and positive semidefinite matrices $F_k \in \R^{(d-k+1) \times (d-k+1)}$, $k = 0,\ldots,d$, the function
\[
(x,y) \mapsto \sum_{k=0}^d  \Big\langle F_k, \overline{Y_k^n}(x\cdot e, y\cdot e, x\cdot y) \Big\rangle 
\]
is a $\mathrm{Stab}_{O(n)}(e)$-invariant positive definite kernel on $S^{n-1}$, and each $\mathrm{Stab}_{O(n)}(e)$-invariant positive definite kernel on $S^{n-1}$ is the uniform limit of such kernels.
\end{proposition}

Define the ball $B^n(R) = \{ x \in \R^n : \|x\| \leq R\}$. The following result is an adaptation of the above result with the sphere $S^{n-1}$ replaced by the ball $B^n(R)$.  Let 
\[
\big(Z_k^n(u,v,t)\big)_{i,i'} = u^{i} v^{i'} (uv)^k P_k^n\left(\frac{t}{uv}\right),
\]  
and set
\[
\overline{Z_k^n}(u,v,t) = \frac{Z_k^n(u,v,t) + Z_k^n(v,u,t)}{2}.
\]

\begin{proposition}\label{prop:balls}
Let $R > 0$. For each integer $d \ge 0$ and positive semidefinite matrices $F_k \in \R^{(d-k+1) \times (d-k+1)}$, $k = 0,\ldots,d$, the function
\[
(x,y) \mapsto \sum_{k=0}^d \Big\langle F_k,  \overline{Z_k^n}(\|x\|, \|y\|, x\cdot y) \Big\rangle 
\]
is an $O(n)$-invariant positive definite kernel on $B^n(R)$, and each $O(n)$-invariant positive definite kernel on $B^n(R)$ is the uniform limit of such kernels.
\end{proposition}
\begin{proof}
By \cite[Theorem A.8]{laat19} we need to show that the $\overline{Z_k^n}$ are the zonal matrices for the space $B^n(R)$ with the action of $O(n)$. For this we define an orthonormal basis $Y^n_{k,j}(x)$, $j \in [d_k]$, for the space of spherical harmonics of degree $k$, which are the homogeneous polynomials of degree $k$ in $n$ variables that vanish under the Laplacian, and we set
\[
e_{k, i, j}(x) = \|x\|^{k+i}\, Y^n_{k,j} \left(\frac{x}{\|x\|}\right), \quad k \geq 0, \, i \geq 0, \, j \in [d_k].
\]
The span of the functions $e_{k,i,j}$ is dense in the space of continuous functions on $B^n(R)$. Moreover, these functions are symmetry adapted in the sense that there exists $O(n)$-equivariant linear maps $T_{i,i'}$ with $T_{k,i,i'} e_{k,i,j} = e_{k,i',j}$. This means these functions form a symmetry adapted system as required by \cite[Theorem A.8]{laat19}, except that the functions are orthonormal for an $O(n)$-invariant inner product as opposed to an $O(n)$-invariant measure on $B^n(R)$, but the proof in \cite{laat19} still works with this weaker condition.

The zonal matrices corresponding to the above symmetry adapted system are defined as
\[
\overline{Z_k^n}(x,y)_{i,i'} = \sum_{j \in d_k} e_{k, i, j}(x) \overline{e_{k, i', j}(y)},
\]
so the proof follows by the addition theorem for spherical harmonics:
\[
\sum_{j \in d_k} Y^n_{k,j}\left(\frac{x}{\|x\|}\right) \overline{Y^n_{k,j}\left(\frac{y}{\|y\|}\right)} = P_k^n\left(\frac{x}{\|x\|}\cdot \frac{y}{\|y\|}\right). \qedhere
\]
\end{proof}

\subsection{Derivation of the semidefinite programming bounds}\label{sec:2and3}

In this paper we give upper bounds for the cardinality of codes in spheres, codes in spherical caps, and codes in balls. The semidefinite programming formulations for spheres and spherical caps are not new, but for completeness we show how these are derived from the general two and three-point bounds. This puts everything, including our new semidefinite programming bounds for codes in balls, in a common framework. In addition it allows us to discuss in detail and supported by computations how the three-point bounds relate to the two-point bounds. 

\subsubsection{Two-point bounds}\label{sec:2pt}

The Lov\'asz $\vartheta$-number is a semidefinite programming upper bound on the independence number of a finite graph, where the \emph{independence number} is the size of a largest subset of the vertices where no two vertices are adjacent. In \cite{bachoc09} this is generalized to the \emph{spherical code graph} $G=(V,E)$, which is the graph with vertex set $V = S^{n-1}$ where two distinct vertices $x$ and $y$ are adjacent if $x \cdot y > \cos(\vartheta)$. This immediately generalizes to a topological packing graph $G = (V,E)$, which is a graph whose vertex set $V$ is a Hausdorff topological space where each finite clique is a subset of an open clique (which in particular forces the independence number to be finite) \cite{laat15}. This gives the optimization problem
\begin{align}\label{eq:theta}
\vartheta(G) = \inf \big\{ M &\in \R : K \in \Ccal(V \times V)_{\succeq 0}\\\nonumber
& K(x,x) \leq M-1 \text{ for } x \in V,\\\nonumber
& K(x,y) \leq -1 \text{ for } \{x,y\} \not\in E \big\}.
\end{align}
By integrating over the symmetry group $\Gamma$ of the graph $G$, one sees that we can restrict to $\Gamma$-invariant functions without affecting the optimal objective value. In \cite{bachoc09} it is shown that for the spherical code graph, Schoenberg's characterization reduces this to the Delsarte linear programming bound \cite{delsarte77}.

To consider $\vartheta$-codes in a spherical cap one just takes a spherical code graph and restricts the vertex set to  $\mathrm{Cap}^{n-1}(e, \varphi)$. By doing this, the symmetry group reduces from $O(n)$ to $O(n-1)$. So instead of Schoenberg's characterization, Bachoc and Vallentin use Proposition~\ref{prop:Bachoc-Vallentin} to derive the following formulation:
\begin{align}\label{eq:hemisphere}
\min \big\{ M &\in \R : F_0,\ldots, F_d \succeq 0,\, \\\nonumber
&\quad F(u,u,1) \leq M-1 \text{ for } \cos(\varphi) \leq u \leq 1,\\\nonumber
&\quad F(u,v,t) \leq -1 \text{ for } (u,v,t) \in \Theta\big\},
\end{align}
where 
\[
F(u,v,t) =\sum_{k=0}^d \left\langle F_k,\overline{Y_k^n}(u,v,t)\right\rangle
\]
and
\begin{align*}
\Theta = \big\{(u,v,t) :\;& \cos(\varphi) \leq u, v \leq 1, \, -1 \leq t \leq\cos(\vartheta),\\\nonumber
&\quad 1+2uvt-u^2-v^2-t^2 \ge 0\big\}.
\end{align*}

Problem \eqref{eq:hemisphere} is not yet a semidefinite program because we have infinitely many inequality constraints. However, since these are polynomial inequality constraints we can formulate, for each integer $\delta > 0$, a semidefinite programming upper bound  by replacing each constraint of the form 
\[
p \geq 0 \quad \text{on} \quad S := \big\{x \in \mathbb{R}^n : g_i(x) \geq 0 \text{ for } i \in [m]\big\},
\]
where $p,g_1,\ldots,g_m$ are polynomials, by the condition that there are sum-of-squares polynomials $q_0$ of degree $2\delta$ and $q_i$ of degree $2\delta-\mathrm{deg}(g_i)$ such that
\[
p(x) = q_0(x) + g_1(x) q_1(x) + \ldots + g_m(x) q_m(x).
\]

This is a semidefinite constraint because if $b_\delta(x)$ is a vector whose entries form a basis for the polynomials in $\R[x_1,\ldots,x_n]$ of total degree at most $\delta$, then a polynomial $q$ of degree $2\delta$ is a sum-of-squares polynomial if and only if there is a positive semidefinite matrix $Q$ such that $q(x) =  \langle Q, b_\delta(x)b_\delta(x)^{\sf T}\rangle$. It follows from Putinar's theorem \cite{putinar93} that we get arbitrary good upper approximations by taking $\delta$ sufficiently large. In the univariate case, for $\deg(p)=2d$, it is sufficient to take $\delta=d$, and in practice we set $\delta=d$ also in the  multivariate case.

Concretely, we replace the first polynomial inequality constraint in \eqref{eq:hemisphere} by
\begin{align*}
&F(u,u,1) - M + 1 + \big\langle Q_1, b_\delta(u)b_\delta(u)^{\sf T} \big\rangle\\
&\quad+ (u - \cos \varphi )(1-u) \big\langle Q_2, b_{\delta-1}(u)b_{\delta-1}(u)^{\sf T} \big\rangle = 0,
\end{align*}
where $Q_1$ and $Q_2$ are positive semidefinite matrices.

To replace the second polynomial inequality constraint in \eqref{eq:hemisphere} we first observe that by construction $F(u,v,t) = F(v,u,t)$ for all $u,v,t$, and we can exploit this symmetry to get a more efficient sums-of-squares characterization. 

We have
\[
\R[u,v,t] = \R[u+v, uv, t] \oplus (u-v) \R[u+v, uv, t],
\]
where $\R[u+v, uv, t]$ is the ring of invariant polynomials. Let $b_\delta=b_\delta(u+v,uv,t)$ be a vector whose entries form a basis for the space of polynomials in $\R[u+v, uv, t]$ of total degree (in the variables $u,v,t$) at most $\delta$. It follows that any sum-of-squares polynomial $q(u,v,t)$ of degree $2\delta$ that satisfies $q(u,v,t) = q(v,u,t)$ for all $u,v,t$ is of the form
\begin{equation}\label{eq:invsos}
q(u,v,t) = \big\langle X_1, b_\delta b_\delta^{\sf T}\big\rangle + (u-v)^2 \big\langle X_2, b_{\delta-1} b_{\delta-1} ^{\sf T}\big\rangle,
\end{equation}
where $X_1$ and $X_2$ are positive semidefinite matrices.

We can now replace the second polynomial constraint in \eqref{eq:hemisphere} by
\begin{align}\label{eq:tripleConstraint_symm}
& F(u,v,t) + 1 + q_1(u,v,t)  \\
&\qquad +[(u - \cos(\varphi))(1-u) + (v - \cos(\varphi))(1-v)]q_2(u,v,t) \nonumber\\
&\qquad+ (u - \cos(\varphi))(1-u)(v - \cos(\varphi))(1-v)q_3(u,v,t)  \nonumber\\
&\qquad+(t + 1)(\cos(\vartheta)-t)q_4(u,v,t) \nonumber\\
&\qquad+ (1+2uvt-u^2-v^2-t^2)q_5(u,v,t) = 0,\nonumber
\end{align}
where each $q_i$ is a sum-of-squares polynomial of the form \eqref{eq:invsos}, with $\delta=d$ for $q_1$, $\delta =d-1$ for $q_2$ and $q_4$, and $\delta=d-2$ for $q_3$ and $q_5$.

Note that this symmetric formulation is quite a bit better than a naive formulation not exploiting the $uv$-symmetry. For example, for the $\pi/3$-code problem in $\mathrm{Cap}^7(e, \pi/2)$ we need to take $d=9$, and by using the symmetries we reduce the the system from $110376$ variables to $37651$ variables.

We now give the direct proof that \eqref{eq:hemisphere} gives a valid upper bound, since we will need the details of this proof in Section~\ref{info-exact-solution}.
\begin{lemma}\label{lemma:wdhemisphere}
The optimal value of \eqref{eq:hemisphere} is an upper bound on the maximal size of a $\vartheta$-code in $\mathrm{Cap}^{n-1}(e,\varphi)$.
\end{lemma}
\begin{proof}
Suppose $C$ is a $\vartheta$-code in $\mathrm{Cap}^{n-1}(e,\varphi)$ and $(M, F_0, \ldots, F_d)$ is a feasible solution to \eqref{eq:hemisphere}.  On the one hand we have 
\[
S := \sum_{k=0}^d \Big\langle F_k, \sum_{x,y \in C} \overline{Y_k^n}(x \cdot e, y \cdot e, x \cdot y)\Big\rangle \geq 0,
\]
since for every $k=0,\ldots,d$, the matrices $F_k$ and $\sum_{x,y \in C} \overline{Y_k^n}(x \cdot e, y \cdot e, x \cdot y)$ are positive semidefinite. On the other hand, since the conditions in \eqref{eq:hemisphere} are satisfied, we have
\begin{align*}
S & = \sum_{x \in C} \sum_{k=0}^d \left\langle F_k,\overline{Y_k^n}(x \cdot e, x \cdot e, 1)\right\rangle + \sum_{\substack{x,y \in C \\x\neq y}} \sum_{k=0}^d \left\langle F_k,\overline{Y_k^n}(x \cdot e, y \cdot e, x \cdot y)\right\rangle
\\ & \leq \sum_{x \in C} (M-1) + \sum_{\substack{x,y \in C \\x\neq y}} (-1)  = |C|(M-1) - |C|(|C|-1)  = |C|(M-|C|).
\end{align*}
The inequality $|C|\leq M$ follows immediately.
\end{proof}

Now we consider the problem of packing spheres of radius $r$ into a sphere of radius $R$. Here we consider the graph with vertex set $B^n(R-r)$, where two distinct vertices $x$ and $y$ are adjacent if $\|x-y\| < 2r$. Using Proposition~\ref{prop:balls}, the optimization problem $\vartheta(G)$ reduces to
\begin{align}\label{sdp:ballsinBalls}
\min\{ M &\in \R : F_0,\ldots, F_d \succeq 0,\\[-0.1em]\nonumber
&F(u,u,u^2) \leq M-1 \text{ for } 0 \leq u \leq R-r,\\\nonumber
&\qquad F(u,v,t)\leq -1 \text{ for } (u,v,t) \in \Omega\},
\end{align}
where
\[
F(u,v,t) =\sum_{k=0}^d \Big\langle F_k,\overline{Z_k^n}(u,v,t)\Big\rangle
\]
and 
\begin{align*}
\Omega = \big\{(u,v,t) :0 \leq u,v \le R-r, \,-uv \le t \le uv, \, u^2 + v^2 -2t - 4r^2 \ge 0\big\}.
\end{align*}
The polynomial inequality constraints can now be replaced by sums-of-squares constraints in the same way as above.

\subsubsection{Three-point bounds}

Fix a point $e \in S^{n-1}$. Since $\mathrm{Cap}^{n-1}(e, \pi) = S^{n-1}$, the linear programming bound \eqref{eq:theta} for spherical codes and the semidefinite programming bound \eqref{eq:hemisphere} for spherical caps, give the same bound when $\varphi = \pi$, but of course \eqref{eq:hemisphere} is much more difficult to compute since it uses less symmetry.

The semidefinite programming bound for spherical caps looks rather similar to the three-point bound for spherical codes; both use Proposition~\ref{geneal-3p-bound}. Before we give the derivation of the three-point bound, we first mention that improved bounds can already be obtained by computing the two-point bound \eqref{eq:hemisphere} for $\mathrm{Cap}^{n-1}(e,\pi-\vartheta)$ and then adding $1$ to the resulting value. This gives an upper bound because we can always rotate a spherical code so that the point $-e$ is in the code. Although this usually only gives a small improvement, in dimension $4$ it gives the upper bound $24.983$ for $d=10$, which shows that the Lov\'asz $\vartheta$-number is actually already enough to prove the optimality of the $24$-cell for the kissing number problem. 

Of course, for spherical code problems it is better to use the three-point bound, because it is equally difficult to compute and gives better bounds.  To get the full three-point bound we have to also derive constraints coming from functions on $V \times V \times V$, where $V = S^{n-1}$. We follow the derivation from \cite{laat18}, where a general formulation for  $k$-point bounds is given. For $k=3$ we get
\begin{align}\label{geneal-3p-bound}
\mathrm{inf}\big\{M &\in \R : T \in \Ccal(V \times V \times I_1)_{\succeq 0},\\\nonumber
& B_3 T(S) \leq M-1 \text{ for } S \in I_{=1},\\\nonumber
& B_3 T(S) \leq -2 \text{ for } S \in I_{=2},\\[-0.1em]\nonumber
& B_3 T(S) \leq 0 \text{ for } S \in I_{=3} \big\},
\end{align}
where $I_{=k}$ is the set of independent sets of cardinality $k$ in the spherical code graph, and $I_k$ is the union of $I_{=i}$ for $0 \leq i \leq k$.
Here $B_3 \colon \Ccal(V \times V \times I_1)_{\mathrm{sym}} \to \Ccal(I_3 \setminus
\{\emptyset\})$ is the operator defined by
\begin{equation}\label{eq:operatorBk}
B_3 T(S) = \sum_{\substack{Q \subseteq S \\ |Q| \leq 1}} \;\sum_{\substack{x, y 
		\in S \\ Q \cup \{x, y\} = S}} T(x, y, Q).
\end{equation}
In \eqref{geneal-3p-bound} we can restrict to $O(n)$-invariant functions. As shown in \cite{laat18}, using Proposition~\ref{prop:Bachoc-Vallentin}, Problem~\eqref{geneal-3p-bound} then reduces to the Bachoc-Vallentin bound: Set
\[
T(x,y,\emptyset) = \sum_{k=0}^d a_k P_k^n(x \cdot y)
\]
and
\[
T(x,y,\{z\}) = \sum_{k=0}^d \Big\langle F_k, S_k^n(x\cdot z, y\cdot z, x\cdot y)\Big\rangle,
\]
where $$S^n_k =\frac{1}{6}\sum_{\sigma \in S_3} \sigma Y^n_k$$ in which we sum over the  permutations on three elements and each element $\sigma \in S_3$ acts on $Y_k^n$ by permuting its arguments. Then, \eqref{geneal-3p-bound} reduces to
\begin{align}\label{eq:3ptsphere}
\min\big\{ M & \in \R : a_0,\ldots,a_d \geq 0, \, F_0,\ldots, F_d \succeq 0,\\[-0.3em]\nonumber
&\sum_{k=0}^d a_k + F(1,1,1) \leq M-1,\\[-0.8em]\nonumber
&\qquad\sum_{k=0}^d a_k P_k^n(u) + 3F(u,u,1) \leq -1 \text{ for } -1 \leq u \leq \cos(\vartheta),\\\nonumber
&\qquad\qquad F(u,v,t) \leq 0 \text{ for } (u,v,t) \in \Delta\big\},
\end{align}
where
\[
F(u,v,t) =\sum_{k=0}^d \Big\langle F_k, S_k^n(u,v,t)\Big\rangle
\]
and 
\[
\Delta = \big\{(u,v,t) : -1 \leq u,v,t \leq \cos(\vartheta), \, 1+2uvt-u^2-v^2-t^2 \ge 0\big\}.
\]

Analogously to the two-point bound, we can use sums-of-squares relaxations to formulate semidefinite programming upper bounds. For this we write the last condition as
\begin{align*}
&F(u,v,t) + q_0(u,v,t) + p(u)q_1(u,v,t)+ p(v)q_2(u,v,t) \\
&\quad + p(t) q_3(u,v,t) + (1+2uvt-u^2-v^2-t^2) q_4(u,v,t) = 0,
\end{align*}
where $p(u) = (\cos(\theta) -u )(1-u)$ and $q_0, \ldots, q_4$ are sum-of-squares polynomials. 

The function $F(u,v,t)$ is symmetric in $u,v,t$, and as in \cite{machado16} we can use this symmetry to give a more efficient characterization using \cite{Gatermann04}. First we reformulate the sums-of-squares factorization as
\begin{align*}
F(u,v,t) + q_0(u,v,t) + \sum_{i=1}^4 s_i q_i(u,v,t)  = 0,
\end{align*}
where
\begin{equation*}
\label{eq:delta-polys}
\begin{aligned}
s_1 &= p(u) + p(v) + p(t),&s_2 &= p(u)p(v) + p(u)p(t) + p(v)p(t),\\
s_3 &= p(u)p(v)p(t),&s_4 &= 1 + 2uvt - u^2 - v^2 - t^2.
\end{aligned}
\end{equation*}
Now, without loss of generality,  we may assume the sum-of-squares polynomials $q_0,\ldots, q_4$ to be symmetric in $u,v,t$. 
Let
$$\theta_1 = u+v+t, ~~\theta_2 = uv+ut+vt, ~~\theta_3  = uvt,$$ 
and let $b_\delta=b_\delta(\theta_1,\theta_2,\theta_3)$  be a vector whose entries form a basis for the space of polynomials in the invariant ring $\R[\theta_1,\theta_2, \theta_3]$ of total degree (in the variables $u,v,t$) at most $\delta$. 
By \cite{Gatermann04}, for each $S_3$-invariant sum-of-squares polynomial $p$ of degree $2\delta$ there are positive semidefinite matrices $Q_1,Q_2,Q_3$ such that
\[
p(u,v,t) = \Big\langle Q_1, b_\delta b_\delta^{\sf T} \otimes \Pi_1 \Big\rangle + \Big\langle Q_2, b_{\delta-3}b_{\delta-3}^{\sf T} \otimes \Pi_2\Big\rangle  + \Big\langle Q_3, b_{\delta-2}b_{\delta-2}^{\sf T} \otimes \Pi_3\Big\rangle
\]
where
\begin{align*}&\Pi_1 = 1, ~~~~~\Pi_2 = \theta_1^2\theta_2^2 - 4\theta_2^3 - 4\theta_1^3\theta_3 + 18\theta_1\theta_2\theta_3 - 27 \theta_3^2,\\
&\Pi_3 = \begin{pmatrix}
2\theta_1^2-6\theta_2 & -\theta_1\theta_2 + 9\theta_3 \\
-\theta_1\theta_2 + 9\theta_3 & 2\theta_2^2 - 6\theta_1\theta_3
\end{pmatrix}.
\end{align*}

\begin{lemma}\label{lem:3pt-bound}
The optimal value of the semidefinite program given in (\ref{eq:3ptsphere}) is an upper bound on the maximal size of a $\vartheta$-code in $C^{n-1}(e, \varphi)$.
\end{lemma}
\begin{proof}
Let $(M, a_0, \ldots, a_d, F_0, \ldots, F_d)$ be a feasible solution of problem~\eqref{eq:3ptsphere}, and let $C$ be a $\vartheta$-code in $S^{n-1}$. 

For each $k \in \{0,\ldots, d\}$, the matrices $\sum_{(x,y,z) \in C^3}\overline{Y_k^n}(x\cdot z, y\cdot z, x\cdot y)\rangle$ and $F_k$ are positive semidefinite, so
$$S := \sum_{(x,y,z) \in C^3} F(x\cdot z, y\cdot z, x\cdot y) =   \sum_{k=0}^d\Big\langle F_k, \sum_{(x,y,z) \in C^3}\overline{Y_k^n}(x\cdot z, y\cdot z, x\cdot y)\Big\rangle \geq 0.$$
On the other hand, $S$ is equal to
\begin{align*}
|C| ~F(1,1,1) +\sum_{\substack{( x,z) \in C^2 \\x\neq z}} 3~F(x\cdot z, x\cdot z, 1) + \sum_{\substack{(x,y,z) \in C^3 \\x\neq z, y \neq z, x\neq y}} F(x\cdot z, y\cdot z, x\cdot y)
\end{align*}
In the above equation, we use that $F(u,v,t)$ is invariant under the permutations of $(u,v,t)$.  
Since $F$ has to satisfy the constraints in (\ref{eq:3ptsphere}), we get
\begin{align*}
S &\leq |C| \left(M-1 - \sum_{k=0}^d a_k\right) - |C|(|C|-1)- \sum_{ x, y\in C} \left(\sum_{k=0}^da_kP^n_k(x\cdot y)\right) \\
 &\leq |C|(M-1)-|C|(|C|-1) - \sum_{k=0}^d a_k\left(|C| + \sum_{\substack{( x,y) \in C^2 \\x\neq y}} P^n_k(x\cdot y)\right) \\
 &\leq |C|(M-1)-|C|(|C|-1) - \sum_{k=0}^d a_k \sum_{ x, y \in C} P^n_k(x\cdot y) \\
&\leq |C|(M-1)-|C|(|C|-1).
\end{align*}
So we have shown $0 \leq S \leq |C|(M-1)-|C|(|C|-1)$, which implies $|C| \leq M$.
\end{proof}

\subsection{Information from exact sharp solutions}\label{info-exact-solution}

Since all our programs give upper bounds on the size of optimal configurations, which has to be an integer, one might wonder why we are interested in exact bounds. Indeed, if we already know a configuration $\mathcal{C}$ of $M$ points, any upper bound strictly lower than $M+1$ ensures the optimality of $\mathcal{C}$. 

First, these bounds are \textit{a priori} only upper bounds on the independence number of the corresponding graph, and it is interesting to point out when these bounds give exactly the independence number. For example, Bachoc and Vallentin proved in \cite{bachoc09b} that for $\vartheta$-codes in four dimensions, where $\cos \vartheta = 1/6$, the three-point bound gives exactly $10$ (even though the two-point bound is not sharp here). 

The second interest is geometric: from the proof of Lemma~\ref{lemma:wdhemisphere}, we can see that any feasible solution reaching $M$ as objective value provides additional information regarding optimal solutions. For any code $C$ such that $|C|=M$, all the inequalities in the proof of Lemma~\ref{lemma:wdhemisphere} have to be equalities. We get the so-called \textit{complementary slackness} conditions:

\begin{corollary}\label{cor:compl1}
Let $(M, F_0, \ldots, F_d)$ be a feasible solution to \eqref{eq:hemisphere}. If $C$ is a $\vartheta$-code in $C^{n-1}(e,\varphi)$ with cardinality $|C|=M$, then the following equalities hold:
\begin{itemize}
\item[i)] for every $k=0,\ldots,d$,
\[
\Big\langle F_k, \sum_{x,y \in C} \overline{Y_k^n}(x \cdot e, y \cdot e, x \cdot y)\Big\rangle = 0,
\]
\item[ii)] for every $x$ in $C$, 
\[
F(x \cdot e, x \cdot e, 1) = M-1,
\]
\item[iii)] for every distinct $x,y$ in $C$, 
\[
F(x \cdot e, y \cdot e, x \cdot y) = -1.
\]
\end{itemize}
\end{corollary}

Conditions ii) and iii) are of special interest. If we define the polynomial
\[
P(u)=F(u, u, 1) - (M-1),
\]
then for every point $x$ in an optimal code, the innerproduct $e \cdot x$ is in the set $R$ of roots of $P$. This gives finitely many possibilities for these innerproducts. Then, for every $u,v\in R$, we can define 
\[
P_{u,v}(t)= F(u, v, t) + 1,
\]
and for every pair of distinct points $(x,y)$ in an optimal code, the innerproduct $x\cdot y$ has to be a root of one such polynomial. With this procedure, we get all the possible innerproducts occuring in the configuration. Once that we get all the possible triples $(u,v,t)$, we can even use i) to get the number of occurences of each triple, by solving a linear system.

If we consider the semidefinite programming bound in \eqref{sdp:ballsinBalls} for the number of spheres that can be packed in a given sphere, we will get the possible norms of the centers of the spheres, and the possible innerproducts among these centers. 

Analogously, three-point bounds also provide complementary slackness conditions. From Lemma~\ref{lem:3pt-bound}, we get:
\begin{corollary}\label{cor:3pt-cs}
Let $(M, a_0, \ldots, a_d, F_0, \ldots, F_d)$  be a feasible solution of the semidefinite program (\ref{eq:3ptsphere}). Let $C$ be a $\vartheta$-code in $S^{n-1}$. If $M = |C|$, then the following properties hold:
\begin{itemize}
\item[i)] $$\sum_{k=0}^d a_k + F(1,1,1) = M-1,$$
\item[ii)] for each two distinct $x,y \in C$, $$\sum_{k=0}^d a_k P_k^n(u) + 3F(x\cdot y, x\cdot y,1) = -1,$$
\item[iii)]for each three distinct $x,y, z \in C$, $$F(x\cdot z, y\cdot z, x\cdot y)  = 0,$$
\item[iv)] for each $k \in \{0, \ldots, d\}$, $$\Big\langle F_k, \sum_{(x,y,z) \in C^3}\overline{Y_k^n}(x\cdot z, y\cdot z, x\cdot y)\Big\rangle = 0,$$
\item[v)] for each $k \in \{0,\ldots, d\}$,  $$ a_k\sum_{x,y \in C} P^n_k(x\cdot y) = 0.$$
\end{itemize}
\end{corollary}

In this situation, the procedure described above also gives the possible innerproducts among points in optimal codes, and the distribution of the triples $(u,v,t)$. 

In several cases, this information turns out to be enough to prove the uniqueness of optimal configurations, as we will see in Section~\ref{sec:applications}. 

On the other hand, if we have a candidate optimal configuration and the bound is sharp, then any optimal solution of the semidefinite program will satisfy the complementary slackness conditions implied by this configuration. Adding these conditions to the semidefinite program usually results in a faster convergence of the interior point method. The complementary slackness conditions can be important to get an exact optimal solution. In the next section, we give an automatic procedure to extract an exact optimal solution from a numerical optimal floating point solution.

\section{The rounding procedure}\label{sec:rouding procedure}

In general, the fact that a semidefinite program is defined over a given algebraic field does not ensure that the optimal solution can be defined over the same field. In fact, even if the constraints and the objective are rational, an optimal solution might require high algebraic degree \cite{MR2546336}. However, for each packing problem that we considered where we have a sharp bound, we were able to find an optimal solution over the same field that is required to define the semidefinite program and to formulate the complementary slackness conditions coming from the optimal solution. We first focus on the problem of finding a rational optimal solution, and extend this to quadratic fields in Section~\ref{sec:algrounding}. 

We work with semidefinite programs in block form:
\begin{align*}
\inf \Big\{ \sum_{i=1}^l \mathrm{tr}(C_i &X_i) : X_1,\ldots,X_l \succeq 0,\\[-0.8em]
&\sum_{i=1}^l \mathrm{tr}(A_{i,j} X_i) = b_j \text{ for } j = 1,\ldots,m\Big\},
\end{align*}
where the symmetric matrices $C_i$ and $A_{i,j}$ and the scalars $b_j$ are all rational. 

First, since we already know the optimal objective value, we do not solve the semidefinite program as an optimization problem, but we add a linear constraint enforcing the objective value and solve the problem as a feasibility problem.  One advantage of this is that by rounding the floating point solution from the semidefinite programming solver to a rational solution, such that all linear equations are satisfied, the objective value will automatically be correct. A second advantage is that the numerical interior point solver we use (sdpa-gmp \cite{SDPA}) gives an approximate solution in the relative interior of the feasible set when solving a feasibility problem, which will be important in Section \ref{sec:kernelDetection}. When rounding we have to make sure that the obtained rational matrices are positive semidefinite and satisfy all linear equations.

\subsection{Rounding in the affine space}\label{sec:RoundingAffineSpace}

Let $X_1^*,\ldots,X_l^*$ be the (high precision) floating point output of the semidefinite programming solver. Because of floating point arithmetic they do not satisfy the linear constraints 
\[
\sum_{i=1}^l \mathrm{tr}(A_{i,j} X_i) = b_j \quad \text{for} \quad j = 1,\ldots,m
\]
exactly. First we want an exact solution to the above equations that is close to the approximate solution $X_1^*,\ldots,X_l^*$. We rewrite the above linear conditions as the linear system
$Ax = b$, so that the concatenation $x^*$ of $\mathrm{vec}(X_i^*)$ for $ i= 1, \ldots, l$ is an approximate solution of this system. Here $\mathrm{vec}$ is the vectorization operator that turns a symmetric matrix into a column vector.

We want to find a vector that satisfies the linear system $Ax=b$ exactly and is close to $x^*$. One natural way to do it is the following: We transform the system $Ax=b$ into reduced row echelon form, which can be done in rational arithmetic. Then, when solving this system by backsubstitution, for each free variable that we encounter, we use a rational approximation (possibly with some upper bound on the size of the denominator) of the corresponding entry of $x^*$. We obtain an optimal rational solution of the semidefinite program that satisfies all linear constraints exactly. 

In our applications, the linear systems we are dealing with can be large, even after exploiting the symmetries in the sums-of-squares formulation. We therefore write the system as a linear system over the integers, and use the Kannan-Bachem algorithm \cite{kannan1979} as implemented in \cite{fieker17}, which can solve the biggest system we consider in this paper within eight hours on a normal desktop computer. 

By continuity of roots, the eigenvalues of the blocks in the rounded solution will be close to the eigenvalues of the blocks in the floating point approximate solution. If the floating point solution would not have near zero eigenvalues, then we would be done. However, in our situation, the matrices $X_i^*$ have many eigenvalues close to zero, so that the rounded solution typically has negative eigenvalues.

Since we already solve the problem as a feasibility problem, the solution we are working with lies in the relative interior of the feasible set of our feasibility problem. This means there exists no solution that has fewer near zero eigenvalues. Even if we somehow manage to obtain a solver output $X_1^*,\ldots,X_l^*$ where all near zero eigenvalue are positive, for example by adding constraints of the form $X-\varepsilon I \succeq 0$, then the above rounding procedure is likely to turn some of these into negative numbers. Instead of trying to make these eigenvalues positive, our strategy consists in forcing them to be exactly zero, by adding some linear constraints. 

\subsection{Kernel detection}\label{sec:kernelDetection}

One approach to find these rational linear constraints is to use the complementary slackness condition from Corollary~\ref{cor:compl1} or Corollary~\ref{cor:3pt-cs} arising from a candidate optimal packing configuration.

In \cite{bachoc09b}, after a slight reformulation of the semidefinite program and some additional work (see Section~\ref{sec:sphericalCodes}), these linear constraints are sufficient to get an exact bound. However in general, especially when the matrices involved are large (see Section~\ref{sec:Hemisphere}), this approach does not provide all the linear constraints, even if we also include constraints coming from the derivatives. We therefore propose an automated procedure to find all necessary linear conditions.

By computing the kernel of $X_i^*$ in high precision floating point arithmetic, we get the eigenvectors corresponding to the near zero eigenvalues. We list these vectors as the columns of the matrix $N_i$. These vectors are typically not approximations of rational vectors themselves, but approximations of real linear combinations of rational vectors, without any control on the coefficients. Hence there is no point in trying to round them to rational vectors. Instead, we try to extract a rational basis of this kernel by searching for integer equations defining this linear space. To do so, we use the function \textsf{lindep} from Nemo \cite{fieker17} that uses the LLL algorithm \cite{MR682664} to find an integer linear combination of the rows of $N_i$ that is close to the zero vector. 
 We remove one row of $N_i$ whose coefficient in this linear combination is nonzero, and repeat the procedure to find another integer linear equation, linearly independent from the previous one. We continue until we found the right number of equations. Let $M_i$ be the matrix with these integer relations as its rows. Then we can compute a basis for the nullspace of $M_i$ in rational arithmetic, and these vectors will be the kernel vectors of the rounded version of the matrix $X_i^*$. 
 
\subsection{Rounding and checking}
For each kernel basis vector for $X_i^*$ that we find we add the constraints $\mathrm{tr}(X_i v) = 0$ to the linear system $Ax=b$.  If there indeed exist rational bases for the kernel of each matrix in the solution, then by performing the rounding procedure mentioned above on this extended semidefinite program, we find positive semidefinite matrices that satisfy all linear constraints exactly. Due to the new linear equations, we make sure that the rounded matrices have no negative eigenvalues.

We have to verify that this is indeed a solution. First the linear constraints can be verified in rational arithmetic. Then we need to verify that the matrices are positive semidefinite. To do so, we compute their characteristic polynomial and use the property that a real-rooted polynomial $f(t)$ has no negative roots if and only if $(-1)^{\deg f} f(-t)$ has only nonnegative coefficients.

Using this approach, we find an optimal rational solution for all considered packing problems where our semidefinite programming bound is sharp and where the complementary slackness conditions implied by an optimal configuration are rational.

\subsection{Extension to quadratic fields}\label{sec:algrounding}

Finally we extend the above approach to the case where we want to round over a quadratic field, which is a natural thing to try whenever the semidefinite program and/or the complementary slackness conditions arising from an optimal configuration are defined over such a field.

We first describe how the rounding part from Section~\ref{sec:RoundingAffineSpace} can be adapted. In this situation, $A$ and $b$ of the linear system $Ax = b$ are defined over $\mathbb Q[\sqrt{\ell}]$. We have an approximate floating point solution $x^*$ and we are looking for an exact solution $x=x_1 + x_2\sqrt{\ell}$ over $\mathbb Q[\sqrt{\ell}]$ that is close to $x^*$. After writing $A = A_1 + A_2 \sqrt{\ell}$ and $b = b_1 + b_2 \sqrt{\ell}$ with the entries of $A_1,A_2,b_1,b_2$ rational, we want to find an exact solution of the following rational linear system
\begin{equation}\label{eq:linearquad}
\begin{pmatrix} A_1 & \ell A_2\\ A_2 & A_1 \end{pmatrix} \begin{pmatrix} x_1\\ x_2\end{pmatrix} = \begin{pmatrix} b_1 \\ b_2 \end{pmatrix}
\end{equation}
such that $x_1 + x_2 \sqrt{\ell} \approx x^*$. To do this, we first need to extract from $x^*$ an approximate solution $(x_1^*, x_2^*)$ of the above linear system such that $x_1^* + x_2^* \sqrt{\ell} \approx x^*$.
To find these vectors we write
\[
x_1^* = y \quad \text{and} \quad x_2^* = \frac{1}{\sqrt{\ell}} (x^* - y)
\]
and solve the linear system
\[
\begin{pmatrix} A_1 - \sqrt{\ell} A_2 \\ \sqrt{\ell}A_2 -  A_1 \end{pmatrix}  y = \begin{pmatrix} b_1 - \sqrt{\ell}A_2 x^*\\ \sqrt{\ell} \, b_2  - A_1 x^*\end{pmatrix} \]
for $y$ in floating point arithmetic.
We now reduce the rational linear system \eqref{eq:linearquad}
to reduced row echelon form and use backsubstitution to find a rational vector satisfying this system, where for each free variable we use a rational approximation of the corresponding entry in $(x_1^*,x_2^*)$.

To find the kernel constraints we again first compute an arbitrary basis for the kernel of $X_i^*$ in high precision floating point arithmetic and list the vectors as the columns of the matrix $N_i$. Then we set
\[
M_i = \begin{pmatrix} N_i \\ \sqrt{\ell} N_i \end{pmatrix}
\]
and use the LLL algorithm to find an integer linear combination $(\lambda, \mu)$ such that $(\lambda, \mu)^{\sf T}M_i \approx 0$. Once we find the right number of equations, we can build the matrix $H$ with rows $(\lambda, \ell \mu)^{\sf T}$ and $(\mu, \lambda)^{\sf T}$, for each equation $(\lambda, \mu)$ that we find. By construction, every vector $(u,v)$ in the kernel of $H$ satisfies $X_i^*(u+ v\sqrt{\ell}) \approx 0$, and we want $u+ v\sqrt{\ell}$ to be in the kernel of the rounded version $X_{i,1} + X_{i,2}\sqrt{\ell}$ of $X_i^*$. So every vector in a basis of the kernel of $H$ provides two equations 
\[
X_{i,1}u + \ell X_{i,2}v=0, \quad X_{i,2}u + X_{i,1}v=0
\]
that we add to the linear system \eqref{eq:linearquad}, taking  into account the block structure. Finally, we can apply the rounding procedure in order to get an optimal solution over $\mathbb{Q}[\sqrt{\ell}]$.

\section{Applications}\label{sec:applications}

In this section we determine exact sharp semidefinite programming bounds for packing problems, and we describe how to prove the uniqueness of some optimal configurations by using the information obtained from complementary slackness.

The code for setting up the semidefinite programs, for running the rounding procedure, and for rigorously checking the rounded solution can be found in the ancillary files from the \textsf{arXiv.org} e-print archive. This program runs in Julia 1.1.0 \cite{Julia} and uses the computer algebra system Nemo \cite{fieker17}. See the included README.txt file for information on how to run the code.

\subsection{Codes in spherical caps}\label{sec:Hemisphere}

In 1979, Odlyzko and Sloane \cite{Odlyzko79}, and independently Levenshtein \cite{levenshtein79} proved that the maximal size of a $\pi/3$-code in $S^7$ is $240$. Such a spherical code $C$ is given by the minimal vectors of the root lattice $\mathsf{E}_8$. Moreover, in 1981 Bannai and Sloane \cite{bannai81} showed that the maximal $\pi/3$-code in $S^7$ is unique up to isometry. Let $e$ be an arbitrary element of $C$, then $C \cap  \mathrm{Cap}^7(e, \pi/2)$ is a $\pi/3$-code on the hemisphere $\mathrm{Cap}^7(e, \pi/2)$ with cardinality $183$. In 2007, Bachoc and Vallentin \cite{bachoc092} prove that the maximal size of such a code is indeed $183$. Here we prove that a $\pi/3$-code in $\mathrm{Cap}^7(e, \pi/2)$ with maximal size is unique up to isometry

\begin{lemma}\label{lemma:InnerproductsHemisphere}
Let $C$ be a $\pi/3$-code in $\mathrm{Cap}^7(e, \pi/2)$ with cardinality $183$. Then: 
\begin{enumerate}
\item[(i)] for every $c\in C$, $$e \cdot c \in \{0, 1/2, 1\},$$
\item[(ii)] for any distinct $c,c' \in C$, 
$$c \cdot c' \in \{- 1, \pm 1/2, 0\}.$$ 
\end{enumerate}
\end{lemma}

\begin{proof} 
Let $(F_0, \ldots, F_9)$ be the exact optimal solution obtained by applying our rounding procedure to the floating point output (available as \textsf{183points.jls} in the arXiv version of this paper) that we obtained by solving \eqref{eq:3ptsphere} for $d = 9$, and consider
\[
F(u,v,t) = \sum_{k=0}^9\langle F_k, \overline{Y^n_k} (u,v,t) \rangle.
\]

Let us first prove (i). Since $C$ in an optimal configuration, following (ii) in Corollary \ref{cor:compl1}, every $c\in C$ has to satisfy
\[
F(e\cdot c, e\cdot c, 1) = 182.
\]
This means that $e \cdot c$ has to be a root of the univariate polynomial $$g(u)=F(u,u,1)-182$$ located in $\Delta_0=[0,1]$. By computing its Sturm sequence, we check that the polynomial $g$ has exactly three distinct roots in the interval $(-\varepsilon, 1]$ for a fixed, arbitrary, small enough $\varepsilon > 0$. Since $g(0)=g(1/2)=g(1)=0$, $g$ cannot have any further roots in $[0,1]$.

In order to prove (ii), consider any distinct $c,c'$ in $C$. Following  (iii) in Corollary \ref{cor:compl1}, we have 
\[
F(e \cdot c, e \cdot c', c \cdot c') = -1.
\]
Due to (i), $e \cdot c$ and $e \cdot c'$ have to be in $\{0, 1/2, 1\}$. As a consequence, $c \cdot c'$  is a root of the univariate polynomial $$h(t)=F(u_0, v_0, t)+1$$ for some $u_0, v_0 \in \{0, 1/2, 1\}$. Moreover $$c \cdot c'
\in\{t\in \R: (u_0,v_0,t)\in \Delta\}.$$ Using the same procedure involving the Sturm sequence, we check that all those possible roots lie in $\{- 1, \pm 1/2, 0\}$. 
\end{proof}

Due to the properties given in Lemma \ref{lemma:InnerproductsHemisphere}, we can prove that $\mathsf{E}_8$ provides the unique optimal $\pi/3$-code in $\mathrm{Cap}^7(e, \pi/2)$.
\begin{theorem}\label{thm:UniquenessHemisphere}
There is, up to symmetry, a unique $\pi/3$-code in $\mathrm{Cap}^7(e, \pi/2)$ with $183$ elements.
\end{theorem} 
\begin{proof}
Let $L_0$ be the additive subgroup of $\R^8$ spanned by $C$. If $x$ and $y$ are two vectors in $L_0$, according to (ii) in Lemma \ref{lemma:InnerproductsHemisphere}, the inner product $\sqrt{2} x \cdot \sqrt{2} y $ is an integer. So the additive subgroup $L=\sqrt{2}L_0$ is an integral lattice.  Moreover since it is spanned by a set of vectors $v$ such that $v \cdot v =2$, $L$ has to be a root lattice. This means it is a direct sum of some irreducible root lattices $\mathsf{A}_d, \mathsf{D}_d, \mathsf{E}_6, \mathsf{E}_7, \mathsf{E}_8 $. Assume that $L= \oplus_{i=1}^k L_i$ for some $k$, where $L_i$ is an irreducible root lattice for every $i=1,\ldots,k$. We denote by $r_i$ the number of roots of $L_i$ and by $d_i$ its rank. The number of roots of the irreducible root lattices is well known (see \cite{conway99}), and if $L_i$ is not $\mathsf{E}_8$, then we have $r_i/d_i < 183/8$. Hence if $L$ is not $\mathsf{E}_8$, its number of roots $r$ satisfies
\[
r=\sum_{i=1}^k r_i = \sum_{i=1}^k d_i\frac{r_i}{d_i} < 183.
\]
So $L=\mathsf{E}_8$. 

There are $240$ roots in $\mathsf{E}_8$. Among the corresponding $240$ points of $L_0$ in $S^7$, let $n_1$ and $n_2$ be the numbers of points lying respectively on the equator $\{ x \in S^7: x\cdot e =0 \}$ of $S^7$ and on the strict upper hemisphere $\{ x \in S^7: x\cdot e > 0 \}$. By symmetry we have $n_1+2n_2=240$. On the other hand, on the same hemisphere we cannot have more than the $183$ points coming from $C$, so $n_1 + n_2 = 183$. Thus the strict upper hemisphere of $S^7$ contains exactly $n_2=57$ elements of $C$. Suppose that $e$ is not in $C$. Then, following (i) in Lemma \ref{lemma:InnerproductsHemisphere},  $57$ elements of the code would lie in $\{ x \in S^7: x\cdot e = 1/2 \}$. This would give a $\vartheta$-code in $S^6$ with $57$ elements, where $\vartheta$ is such that $\cos \vartheta = 1/3$. Bannai and Sloane \cite[Theorem 9]{bannai81} proved that such a code does not exist. So $e$ has to be an element of $C$, and $C$ is the configuration that we expect.
\end{proof}

One might wonder whether this approach can be used in dimension $24$. However, whereas the Leech lattice $\Lambda_{24}$ provides a configuration with $144855$ points on the hemisphere, the best upper bound that we obtained numerically is only $158611$, with $d=13$. 

\subsection{Spherical codes}\label{sec:sphericalCodes}

In \cite{bachoc09b}, Bachoc and Vallentin proved the following theorem.

\begin{theorem}\label{thm:PetersenCode}
The Petersen code is, up to symmetry, the unique $\vartheta$-code in $S^3$ of cardinality $10$, where $\vartheta$ is such that $\cos \vartheta = 1/6$. 
\end{theorem}

Their approach, based on semidefinite optimization, consists of three steps: First, they provide an exact optimal solution for a slightly different formulation of problem (\ref{eq:3ptsphere}). From this solution, they derive the three points distance distribution of an optimal code. Finally they prove that the Petersen code is the only code satisfying this distribution. However, the way they obtain their exact optimal solution is not straightforward. First of all, their new formulation of problem (\ref{eq:3ptsphere}) gives further information about the solution matrices $(F_0, F_1, F_2)$, in case they provide an optimal solution. These properties help to determine an exact optimal solution of their semidefinite program, after several steps and various computations. However the strategy that they use seem to apply only in this specific situation. Especially, because for this problem the corresponding solution of problem (\ref{eq:3ptsphere}) is small, since taking $d = 2$ is sufficient. 

In our framework, the first step of their approach gets much easier: we can directly solve problem (\ref{eq:3ptsphere}), and turn the approximate solution from the solver into an exact optimal solution by using our rounding procedure. Then we need to check that our solution also implies the three points distance distribution of an optimal code.

\begin{proof}[Proof of Theorem \ref{thm:PetersenCode}]

Let $C$ be a $\vartheta$-code in $S^3$ with cardinality $10$, where $\vartheta$ is such that $\cos \vartheta = 1/6$. The three points distance distribution of $C$ is defined by 
\[
\alpha(u,v,t)= \frac{1}{|C|} \left|\left\{ (c,c',c'')\in C^3: c\cdot c'=u, c\cdot c''=v, c '\cdot c''=t \right\}\right|.
\]

In order to determine this distribution, we first need to know the possible inner products between the elements of $C$. Let $(M,a_0,\ldots,a_d,F_0,\ldots,F_d)$ be the exact optimal solution obtained by applying our rounding procedure to the floating point output (available as \textsf{10points.jls} in the arXiv version of this paper) that we obtained by solving \eqref{eq:3ptsphere} for $d = 6$. Using the complementary slackness condition given in (ii) in Corollary \ref{cor:3pt-cs}, the inner product between two distinct elements of $C$ must be a root of the polynomial 
\[
\sum_{k=0}^d a_k P_k^n(u) + 3\sum_{k=0}^d \langle F_k,\overline{Y_k^n}(u,u,1)\rangle + 1
\] located between $-1$ and $1/6$. By using a Sturm sequence, we check that the only possible roots are $-2/3$ and $1/6$. 

Now, due to symmetries, the three points distance distribution of $C$ is defined by the six values 
\begin{center}
\begin{tabular}{lll}
 $\alpha(1,1,1)$, & $\alpha(1/6,1/6,1)$, & $\alpha(-2/3,-2/3,1)$, \\  $\alpha(1/6,1/6,1/6)$, & $\alpha(-2/3,1/6,1/6)$, & $\alpha(-2/3,-2/3,1/6)$.
\end{tabular}
\end{center}

We already know $\alpha(1,1,1)=1$. By combining equation $\sum_{u} \alpha(u,u,1) = 10$ with the complementary slackness conditions given by v) in Corollary \ref{cor:3pt-cs}, we get $\alpha(1/6,1/6,1)=6$ and $\alpha(-2/3,-2/3,1)=3$. Then by solving the linear system involving $\sum_{u,v,t} \alpha(u,v,t) = 100$ together with the equations given by iv) in Corollary \ref{cor:3pt-cs}, we get
\begin{center}
\begin{tabular}{lll}
  $\alpha(1/6,1/6,1/6)=18$, & $\alpha(-2/3,1/6,1/6)=12$, & $\alpha(-2/3,-2/3,1/6)=6$.
\end{tabular}
\end{center}

So we recovered the distribution described in \cite{bachoc09b}, and we can apply their last argument to prove that $C$ has to be the Petersen code.
\end{proof}

The optimal value of the program (\ref{eq:3ptsphere}) is an upper bound on the cardinality of a $\vartheta$-code in $S^2$ where $\vartheta$ is such that $\cos \vartheta = (2\sqrt{2}-1)/7$. By using our rounding approach we obtain an exact solution with value $8$.  Sch\"utte and van der Waerden \cite{Schuette51} proved that this is the optimal value. 
The \emph{square antiprism} provides a spherical code with $8$ points in $S^2$. Due to Danzer \cite{Danzer86} this code is unique up to symmetry. His proof rely on heavy geometric arguments: For example, understanding the angles and distances which may occur in such a configuration requires a very technical analysis. Here we recover this information thanks to the complementary slackness conditions, and provide a new proof of the uniqueness of such a code. 

\begin{theorem}\label{thm8pts}
The largest $\vartheta$-code in $S^2$ with $\cos \vartheta = (2\sqrt{2}-1)/7$ is unique up to symmetry.
\end{theorem}

\begin{proof}
Let $(M,a_0,\ldots,a_d,F_0,\ldots,F_d)$ be the exact optimal solution obtained by applying our rounding procedure to the floating point output (available as \textsf{8points.jls} in the arXiv version of this paper) that we obtained by solving \eqref{eq:3ptsphere} for $d = 7$. Furthermore, let $C$ be an optimal spherical code. Due to the complementary slackness condition given in (ii) in Corollary \ref{cor:3pt-cs}, the real roots of the polynomial
\[
\sum_{k=0}^d a_k P_k^n(u) + 3\sum_{k=0}^d \langle F_k,\overline{Y_k^n}(u,u,1)\rangle + 1
\] located between $-1$ and $\frac{2\sqrt{2}-1}{7}$ give the possible inner products between two points in $C$. By using its Sturm sequence, we can check that the only inner products are $$(2\sqrt{2}-1)/7, ~~-3(2\sqrt{2}-1)/7, ~~-(2\sqrt{2}-1)^2/7.$$

As in the proof of Theorem \ref{thm:PetersenCode}, the exact solution, together with the complementary slackness condition v) in Corollary \ref{cor:3pt-cs} and the properties of the three points distribution, gives
\begin{equation*}
\begin{aligned}
\alpha\left(1,1,1  \right) &= 1, &\alpha\left(v, v,1\right) &= 4, \\
\alpha\left(-3v, -3v,1\right) &= 2, &\alpha\left((1-2\sqrt{2})v,(1-2\sqrt{2})v ,1\right) &= 1.
\end{aligned}
\end{equation*}
where $v = (2\sqrt{2}-1)/7$.

This is already enough information for proving the uniqueness of the configuration. We want to show that $C$ is, up to rotations, the squared antiprism $S_8$. We label with $c_1,\ldots,c_8$ the vertices of $S_8$, as depicted in Figure~\ref{fig:8pt}. 

\begin{figure}[!h]
\includegraphics[scale=0.6]{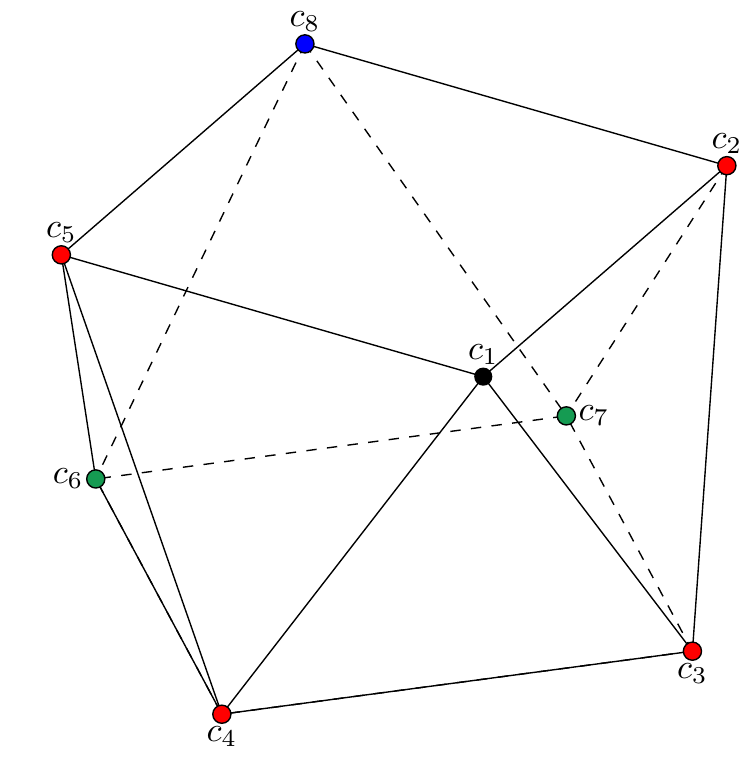}
\caption{The squared antiprism.}\label{fig:8pt}
\end{figure}

Up to symmetry, we may assume that $c_1$ is in $C$. The possible inner products imply that for any $c \in C$, the distance between $c$ and $c_1$ is among
\[
\delta_1= \sqrt{2-2v}, \quad \delta_2= \sqrt{2-2(1-2\sqrt{2})v}, \quad \delta_3= \sqrt{2+6v}.
\]
Note that $\delta_1 < \delta_2 < \delta_3$. In Figure~\ref{fig:8pt}, the points $c_2$, $c_3$, $c_4$ and $c_5$, depicted in red, are at distance $\delta_1$ from $c_1$, the blue point $c_8$ is at distance $\delta_2$ from $c_1$, and the green points $c_6$ and $c_7$ are at distance $\delta_3$ from $c_1$.
For every $c\in C$, we define the circle $\mathcal{C}_{\delta_1}(c)$ made of the points of $S^2$ at distance $\delta_1$ from $c_1$. On the one hand, since $\alpha(v,v,1)=4$, on average over $C$, the circle $\mathcal{C}_{\delta_1}(c)$ contains four points of $C$. On the other hand, it is straightforward to check that one cannot put more than four points on such a circle without violating the distance constraints. This means that for every $c\in C$, there are exactly four elements of the code $C$ in $\mathcal{C}_{\delta_1}(c)$.

In particular, $\mathcal{C}_{\delta_1}(c_1)$ contains four points of $C$. In fact, there is, up to rotations, only one way to put four compatible points on such a circle. Hence, me way assume that these four points are $c_2$, $c_3$, $c_4$, and $c_5$ (see Figure~\ref{fig:circle}).

\begin{figure}[!h]
\includegraphics[scale=0.7]{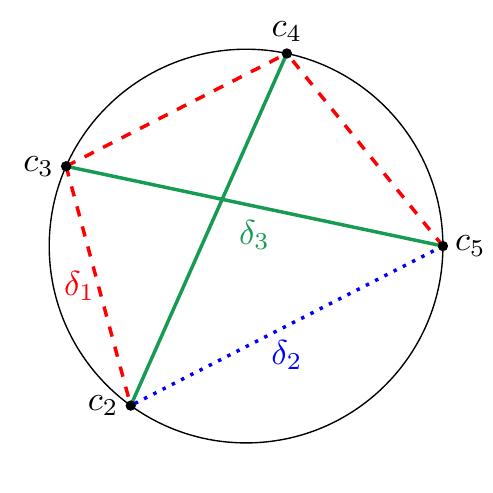}
\caption{The only four point configuration in $\mathcal{C}_{\delta_1}(c)$.}\label{fig:circle}
\end{figure}

Three points remain. On the circle $\mathcal{C}_{\delta_1}(c_3)$, we already know three points out of four. This leaves two possibilities for the last codeword on that circle, and the same situation holds in $\mathcal{C}_{\delta_1}(c_4)$. This gives four candidates, but among them, only two are compatible: $c_6$ and $c_7$. Finally, the last remaining point has to lie in the intersection of $\mathcal{C}_{\delta_1}(c_2)$ and $\mathcal{C}_{\delta_1}(c_5)$, and the distance constraints force it to be $c_8$. 
\end{proof}

\subsection{Sphere in sphere packings}\label{sec:ballsinballs}

The optimal solution of the semidefinite program \eqref{sdp:ballsinBalls} is an upper bound on the number of spheres of radius $r$ that can be packed into a sphere of radius $R$.  
Using our program we determine exact sharp solutions for several values of $r$ and $R$ in various dimensions. In this section we show that we can extract bounds that work for infinitely many different dimensions by extrapolating from the rounded solutions for a few dimensions.  In this way we show that the  Lov\'asz $\vartheta$-number gives a sharp bound  for some families of problems.

\begin{theorem}\label{thm:spheresInspheres}
The Lov\'asz $\vartheta$-number gives a sharp bound on the largest number $M$ of $n$-dimensional unit spheres that can be packed into a sphere of radius $R$, for
\begin{itemize}
\setlength\itemsep{0.3em}
\item[(i)] $n\geq 2$ with $R = 2$ and $M = 2$;
\item[(ii)] $n\geq 2$ with $R = 2/\sqrt{3}+1$ and $M = 3$;
\item[(iii)] $n\geq 2$ with $R = \smash{\sqrt{2n/(n+1)}} + 1$ and $M =n+1$;
\item[(iv)] $n \geq 2$ with $R = \smash{\sqrt{2}}+1$ and $M=2n$;
\item[(v)] $n = 2$ with $R = 1 + \smash{\sqrt{2\left(1 + 1/\sqrt{5}\right)}}$ and $M = 5$;
\item[(vi)] $n = 2$ with $R = 3$ and $M = 7$.
\end{itemize}       
\end{theorem}

To obtain the proof for (i)-(iv)	 we first compute an exact sharp solution for several dimensions. From this we notice we can take the same degree $d$ for each dimension $n$, and we make a guess as to what the general solution of the semidefinite program should be as a function of $n$. For the sums-of-squares matrices we then compute the Cholesky decomposition using symbolic mathematics, which gives us the sums-of-squares decomposition proving that the solution satisfies the right inequalities. Even if we stated the theorem in an unified way with $r=1$, we present some of the proofs with different scaling, in order to simplify the coefficients appearing in the solutions.
\begin{proof}
(i) Set $M = 2$. Define the positive semidefinite matrices
\[
F_0 = \frac{2}{5} \begin{pmatrix} 1 & -1 \\ -1 & 1 \end{pmatrix}, 
\quad  F_1 = \begin{pmatrix} 1 \end{pmatrix},
\]
and set
\begin{align*}
F(u,v,t) := \sum_{k=0}^1 \left\langle F_k,\overline{Z_k^n}(u,v,t)\right\rangle = \frac{2}{5}\big( uv - u - v + t + 1\big).
\end{align*}
Since 
\begin{align*}
F(u,v,t) = - 1 &~ - \frac{3}{5}\left(1 - \frac{1}{2}\left(u + v\right)\right)^2 - \frac{7}{20}\left(v-u\right)^2 \\ 
&- u\left(1 - u\right) - v\left(1 - v\right) - \frac{1}{2}\left(u^2+v^2 -2t - 4\right) 
\end{align*}
we have $F(u,v,t) \leq -1$ for all $u,v,t$ with $0 \leq u,v \leq 1$ and $u^2+v^2-2t -4 \geq 0$.
Furthermore, since
\[
F(u,u,u^2) = 1 - \frac{3}{5}\left(1 - u\right)^2 - 2u\left(1-u\right),
\]
we have
$
F(u,u,u^2) \leq M-1$ for all $0 \leq u \leq 1$. 
The tuple $(M, F_0, F_1)$ thus defines a feasible solution for problem (\ref{sdp:ballsinBalls}) with $R = 2$ and $r=1$, which shows the bound is sharp.\\

(ii) We construct a feasible solution for problem \eqref{sdp:ballsinBalls} for $n \geq 2$ with $r = 2/\sqrt{3}$ and $R = (2/\sqrt{3}+1)r = 4/3 + 2/\sqrt{3}$: By considering the positive semidefinite matrices
\[
F_0 = \begin{pmatrix} 64/81 & -16/27 \\ -16/27 & 4/9 \end{pmatrix}, 
\quad  F_1 = \begin{pmatrix} 9/8 \end{pmatrix},
\]
we obtain the function
\begin{align*}
F(u,v,t) = \sum_{k=0}^1 \left\langle F_k,\overline{Z_k^n}(u,v,t)\right\rangle = \frac{64}{81} + \frac{4}{9}uv  - \frac{16}{27} (u + v) + \frac{9}{8}t.
\end{align*}
Since the function satisfies the equation 
\begin{align*}
F(u,v,t)  = &~-1 -  2\left(7/9 - 7/24\left(u + v\right)\right)^2 - 113/288\left(v-u\right)^2 \\ 
&- 9/8u\left(4/3 - u\right) - 9/8v\left(4/3 - v\right) - 9/16\left(u^2+v^2 -2t -16/3\right) .
\end{align*}
we have
$F(u,v,t) \leq -1$
for all $u,v,t$ with $0 \leq u,v \leq 4/3$ and $u^2+v^2-2t -16/3 \geq 0$.
Furthermore,
$$F(u,u,u^2) =  2 -2\left(7/9 - 7/12u\right)^2 - 9/4u\left(4/3-u\right),$$
so, for all $0 \leq u \leq 4/3,$ we have $F(u,u,u^2) \leq 2$.
The function $F(u,v,t)$ together with the parameter $M=3$ gives a feasible solution for problem (\ref{sdp:ballsinBalls}) with $r=2/\sqrt{3}$ and $R = (2/\sqrt{3}+1)r$.
\\
%, wich gives the result after rescaling by $1/r$. \\

(iii) We construct a feasible solution for problem \eqref{sdp:ballsinBalls} with $r=\sqrt{2n/(n+1)}$, $R = (\sqrt{2n/(n+1)}+1)r = 2n/(n+1) + \sqrt{2n/(n+1)}$ and $n \geq 2$: By considering the positive semidefinite matrices
\[
F_0 = \begin{pmatrix} 2n/(n+1) & -1 \\ -1 & (n+1)/(2n) \end{pmatrix}, 
\quad  F_1 = \begin{pmatrix} \frac{(n+1)^2}{4n} \end{pmatrix},
\]
we define the function 
\begin{align*}
F(u,v,t) = \sum_{k=0}^1 \left\langle F_k^n,\overline{Z_k^n}(u,v,t)\right\rangle = \frac{2n}{n+1} - u - v + \frac{n+1}{2n}uv  + \frac{(n+1)^2}{4n}t.
\end{align*}
Since the following equation holds
\begin{align*}
-F(u,v,t) -1 = &~ \left(\sqrt{\frac{n(n-1)}{n+1}} + \frac{1}{4}\sqrt{\frac{n^2-1}{n}} (u+v) \right)^2 + \frac{n^2 +4n +3}{16n}(v-u)^2 \\ 
&+ u\left(\frac{2n}{n+1} - u\right)\frac{(n+1)^2}{4n} + v\left(\frac{2n}{n+1} - v\right)\frac{(n+1)^2}{4n} \\
&+ \left(u^2+v^2 -2t - 4\frac{2n}{n+1}\right) \frac{(n+1)^2}{8n},
\end{align*}
we have $F(u,v,t) \leq -1$ for all $u,v,t$ with $0 \leq u,v \leq \frac{2n}{n+1}$ and $u^2+v^2-2t -4\frac{2n}{n+1} \geq 0$.
Moreover, 
$$n - F(u,u,u^2) =  \left(\sqrt{\frac{n(n-1)}{n+1}} - \sqrt{\frac{n^2-1}{4n}} u \right)^2 + u\left(\frac{2n}{n+1} - u\right)\frac{(n+1)^2}{2n},$$
so $F(u,u,u^2) \leq n$ for all  $0 \leq u \leq 2n/(n+1)$.

The function $F(u,v,t)$ together with the parameter $M=n+1$ gives a feasible solution for problem (\ref{sdp:ballsinBalls}) with $r=\sqrt{2n/(n+1)}$ and $R = (\sqrt{2n/(n+1)}+1)r$.
\\

(iv) The cases $n=2,3,4$ are done separately through our rounding procedure; see \textsf{proofs.jl} in the arXiv version of this paper. For $n \ge 5$ we determine a feasible solution for \eqref{sdp:ballsinBalls} with $r=\sqrt{2}$ and $R = (\sqrt{2} + 1)r = 2 + \sqrt{2}$. For this we define the positive semidefinite matrices
\[
F_0 = \begin{pmatrix}0 & 0 & 0 \\ 0 & 0 & 0 \\ 0 & 0 & 0 \end{pmatrix}, 
\quad F_1 = \begin{pmatrix} n & -\frac{1}{4}n \\ -\frac{1}{4}n & \frac{1}{16}n \end{pmatrix}, 
\quad  F_2 = \begin{pmatrix} \frac{1}{16}(n-1) \end{pmatrix},
\]
and set  
\begin{align*}
F(u,v,t) &= \sum_{k=0}^2 \left\langle F_k,\overline{Z_k^n}(u,v,t)\right\rangle\\
&= nt - \frac{1}{4}nt(u+v) + \frac{1}{16}nuvt  + \frac{1}{16}nt^2 - \frac{1}{16}u^2v^2.
\end{align*}
To prove that $F(u,v,t) \leq -1$ for all $(u,v,t) \in \Omega$ we want to find sum-of-squares polynomials $q_1(u,v,t), \ldots, q_5(u,v,t)$ such that $-1-F(u,v,t)$ can be written as
\begin{align*}
q_1(u,v,t) &+ u(2-u)q_2(u,v,t) + v(2-v)q_3(u,v,t) \\
&+ (t+uv)(uv-t)q_4(u,v,t) + (u^2+v^2-2t-8) q_5(u,v,t).
\end{align*}
Since $0 \leq u,v \leq 2$ the condition $-uv \leq t \leq uv$ in $\Omega$ holds if $-4 \leq t \leq 4$. Now we can show that $F(u,v,t) \leq -1$ for all $(u,v,t) \in \Omega$, since $-1-F(u,v,t)$ is equal to 
\begin{align*}
q_1(u,v,t) &+ u(2-u)q_2(u,v,t) + v(2-v)q_3(u,v,t) \\
&+ (t+4)(4-t)q_4(u,v,t) + (u^2+v^2-2t-8) q_5(u,v,t).
\end{align*}
where
\begin{align*}
&q_1(u,v,t) =\frac{1}{14}\left(\frac{35}{144}t^2 - \frac{151}{432}tu - \frac{125}{432}tv + \frac{9}{4}t\right)^2\\
&\quad + \frac{25}{238}\left(\frac{25}{144}t^2 + \frac{31}{432}tu + \frac{119}{432}tv\right)^2 + \frac{145}{238}\left(\frac{1}{16}t^2 + \frac{1}{8}tu\right)^2 \\
&\quad+ \frac{1}{21}\left(\frac{5}{16}tu + \frac{25}{16}tv + \frac{7}{4}uv - \frac{15}{4}t - \frac{7}{2}u\right)^2  \\
&\quad+ \frac{n-5}{2}\biggl(\frac{5}{192}\left(3t^2 - 2tu + 8tv\right)^2 
	+ \frac{1}{48}\left(3t^2 - 8tu - 10tv + 48t\right)^2  \\
&\quad+ \frac{21}{64}\left(t^2 + 2tu\right)^2 + \frac{1}{4}\left(tu + 5tv + 8uv - 12t - 16u\right)^2\biggr),\\
\end{align*}
\begin{align*}
&q_2(u,v,t) =\frac{1}{7982}\left(\frac{65}{8}t + \frac{307}{12}u\right)^2 + \frac{2095}{117888}t^2 + \frac{1}{858}\left(\frac{63}{8}u - \frac{65}{16}v\right)^2 \\
&\quad+ \frac{1}{22}\left(\frac{5}{24}u - \frac{25}{16}v + \frac{11}{2}\right)^2+ (n-5)\biggl(\frac{1}{3110}\left(\frac{15}{8}t + \frac{311}{36}u\right)^2   \\
&\quad+ \frac{1461}{622}\left(\frac{1}{24}t\right)^2+ \frac{1}{10}\left(\frac{19}{72}u - \frac{5}{16}v\right)^2+ \frac{3}{2}\left(\frac{1}{72}u - \frac{5}{48}v + \frac{1}{2}\right)^2\biggr),\\
%\end{align*}
%\begin{align*}
&q_3(u,v,t) = \frac{7}{51}\left(\frac{15}{16}t + \frac{17}{16}u\right)^2 + \frac{1}{2}\left(\frac{1}{8}t + \frac{5}{16}u - \frac{5}{12}v\right)^2\\
&\quad + \frac{5}{2}\left(\frac{1}{16}u + \frac{1}{12}v - \frac{1}{2}\right)^2 + \frac{n-5}{430}\biggl(30\left(\frac{7}{16}t + \frac{43}{48}u\right)^2\\
&\quad + 43\left(\frac{1}{8}t + \frac{5}{16}u - \frac{5}{12}v\right)^2 + \frac{1837}{384}t^2 \\
&\quad+ 215\left(\frac{1}{16}u + \frac{1}{12}v - \frac{1}{2}\right)^2\biggr) + \frac{23}{51}\left(\frac{1}{16}t\right)^2,\\
%\end{align*}
%\begin{align*}
&q_4(u,v,t) = \frac{n-5}{2}\left(\left(\frac{1}{16}t\right)^2 + \left(\frac{13}{96}u + \frac{11}{96}v - \frac{1}{2}\right)^2 + 23\left(\frac{1}{96}\left(u - v\right)\right)^2\right)\\
&\quad + \frac{5}{512}t^2 + \left(\frac{53}{288}u + \frac{55}{288}v - \frac{3}{4}\right)^2 + 215\left(\frac{1}{288}\left(u - v\right)\right)^2,\\
%	\end{align*}
%\begin{align*}
&q_5(u,v,t) = n\left(\left(\frac{1}{16}t - \frac{1}{24}u - \frac{1}{12}v + \frac{1}{2}\right)^2 + \left(\frac{1}{16}t + \frac{1}{8}u\right)^2 + \frac{1}{72}\left(u - v\right)^2\right).
	\end{align*} 
Next, we show $F(u,u,u^2) \leq 2n-1$ for all $0 \leq u \leq 2$. For this we define $$ f(u) = F(u,u,u^2) -2n + 1 = \frac{1}{16}\left(2n -  1\right)u^4 - \frac{1}{2}nu^3 + nu^2 - 2n + 1.$$ Note that $f(2) = 0$. It is then sufficient to prove that the polynomial
\[
g(u) = f(u) / (u-2) = \frac{2n-1}{16} u^3 + \frac{-2n-1}{8} u^2 + \frac{2n-1}{4}u + \frac{2n-1}{2}
\]
is positive for every $u$ in $[0,2]$.
 Its discriminant
$$\frac{1}{16}\left(-11n^4 + 21n^3 - 47/4n^2 + 5/2n - 1/4 \right)$$
 is negative, since $-11n^4 + 21n^3$ and $-\frac{47}{4}n^2 + \frac{5}{2}n$ are both negative. Hence $g(u)$ has only one real root. This root must be negative, because the leading term of $g$ is positive and $g(0)=n-1/2>0$. Thus $g$ is positive for every $u\geq 0$. In particular $f(u)\leq 0$ for all $u$ in $[0,2]$, for every $n\geq 5$. 

The function $F(u,v,t)$ together with the parameter $M=2n$ gives a feasible solution for problem (\ref{sdp:ballsinBalls}) with $r=\sqrt{2}$ and $R = (\sqrt{2} + 1)r$.
\\

(v/vi) The exact optimal solution as well as the verification of the solution, can be obtained by running our program as described in \textsf{proofs.jl}. Here the exact solution for (v) is over $\mathbb{Q}[\sqrt{5}]$ and for (vi) over $\mathbb{Q}$.
\end{proof}

\section{Acknowledgements}
We are very thankful to Christine Bachoc, Henry Cohn, Mathieu Dutour Sikiri\'c, Frank Vallentin, and Guillermo Fernández Castro for helpful discussions and suggestions. This material is based upon work supported by the National Science Foundation under Grant No.  DMS-1439786 while the authors were in residence at the Institute for Computational and Experimental Research in Mathematics in Providence, RI, during the ``Point Configurations in Geometry, Physics and Computer Science'' semester program. P.M. received support from the Troms\o \ Research Foundation grant 17matteCR.


\begin{thebibliography}{10}

\bibitem{bachoc09}
Christine Bachoc, Gabriele Nebe, Fernando~M. de~Oliveira~Filho, and Frank
  Vallentin.
\newblock Lower bounds for measurable chromatic numbers.
\newblock {\em Geom. Funct. Anal.}, 19(3):645--661, 2009.

\bibitem{bachoc08}
Christine Bachoc and Frank Vallentin.
\newblock New upper bounds for kissing numbers from semidefinite programming.
\newblock {\em J. Amer. Math. Soc.}, 21(3):909--924, 2008.

\bibitem{bachoc09b}
Christine Bachoc and Frank Vallentin.
\newblock Optimality and uniqueness of the {$(4,10,1/6)$} spherical code.
\newblock {\em J. Combin. Theory Ser. A}, 116(1):195--204, 2009.

\bibitem{bachoc092}
Christine Bachoc and Frank Vallentin.
\newblock Semidefinite programming, multivariate orthogonal polynomials, and
  codes in spherical caps.
\newblock {\em European J. Combin.}, 30(3):625--637, 2009.

\bibitem{bannai04}
Eiichi Bannai, Akihiro Munemasa, and Boris Venkov.
\newblock The nonexistence of certain tight spherical designs.
\newblock {\em Algebra i Analiz}, 16(4):1--23, 2004.

\bibitem{bannai81}
Eiichi Bannai and Neil J.~A. Sloane.
\newblock Uniqueness of certain spherical codes.
\newblock {\em Canad. J. Math.}, 33(2):437--449, 1981.

\bibitem{Julia}
Jeff Bezanson, Alan Edelman, Stefan Karpinski, and Viral~B. Shah.
\newblock Julia: a fresh approach to numerical computing.
\newblock {\em SIAM Rev.}, 59(1):65--98, 2017.

\bibitem{MR4072646}
Diego Cifuentes, Thomas Kahle, and Pablo Parrilo.
\newblock Sums of squares in {M}acaulay2.
\newblock {\em J. Softw. Algebra Geom.}, 10(1):17--24, 2020.

\bibitem{Cohn12}
Henry Cohn and Jeechul Woo.
\newblock Three-point bounds for energy minimization.
\newblock {\em J. Amer. Math. Soc.}, 25(4):929--958, 2012.

\bibitem{conway99}
John~H. Conway and Neil~J.A. Sloane.
\newblock {\em Sphere packings, lattices and groups}, volume 290 of {\em
  Grundlehren der Mathematischen Wissenschaften [Fundamental Principles of
  Mathematical Sciences]}.
\newblock Springer-Verlag, New York, third edition, 1999.

\bibitem{Danzer86}
Ludwig~W. Danzer.
\newblock Finite point-sets on {$S^2$} with minimum distance as large as
  possible.
\newblock {\em Discrete Math.}, 60:3--66, 1986.

\bibitem{laat19}
David de~Laat.
\newblock Moment methods in energy minimization: {N}ew bounds for {R}iesz
  minimal energy problems.
\newblock {\em Trans. Amer. Math. Soc.}, 2019.

\bibitem{laat18}
David {de Laat}, Fabr{\'\i}cio {Caluza Machado}, Fernando {M. de Oliveira
  Filho}, and Frank {Vallentin}.
\newblock {$k$-point semidefinite programming bounds for equiangular lines}.
\newblock {\em arXiv e-prints}, page arXiv:1812.06045, Dec 2018.

\bibitem{laat15}
David de~Laat and Frank Vallentin.
\newblock A semidefinite programming hierarchy for packing problems in discrete
  geometry.
\newblock {\em Math. Program.}, 151(2, Ser. B):529--553, 2015.

\bibitem{delsarte77}
Philippe Delsarte, Jean-Marie Goethals, and Johan~J. Seidel.
\newblock Spherical codes and designs.
\newblock {\em Geometriae Dedicata}, 6(3):363--388, 1977.

\bibitem{dostert17}
Maria Dostert, Crist\'{o}bal Guzm\'{a}n, Fernando~M. de~Oliveira~Filho, and
  Frank Vallentin.
\newblock New upper bounds for the density of translative packings of
  three-dimensional convex bodies with tetrahedral symmetry.
\newblock {\em Discrete Comput. Geom.}, 58(2):449--481, 2017.

\bibitem{fieker17}
Claus Fieker, William Hart, Tommy Hofmann, and Fredrik Johansson.
\newblock Nemo/{H}ecke: computer algebra and number theory packages for the
  {J}ulia programming language.
\newblock In {\em I{SSAC}'17---{P}roceedings of the 2017 {ACM} {I}nternational
  {S}ymposium on {S}ymbolic and {A}lgebraic {C}omputation}, pages 157--164.
  ACM, New York, 2017.

\bibitem{SDPA}
Katsuki Fujisawa, Mituhiro Fukuda, Kazuhiro Kobayashi, Masakazu Kojima,
  Kazuhide Nakata, Maho Nakata, and Makoto Yamashita.
\newblock {SDPA} (semidefinite programming algorithm) user's manual -- version
  7.0.5.
\newblock Technical report, Research Report B-448, Dept.~of Mathematical and
  Computing Sciences, Tokyo Institute of Technology, Tokyo, 2008.

\bibitem{Gatermann04}
Karin Gatermann and Pablo~A. Parrilo.
\newblock Symmetry groups, semidefinite programs, and sums of squares.
\newblock {\em J. Pure Appl. Algebra}, 192(1-3):95--128, 2004.

\bibitem{MR3574590}
Didier Henrion, Simone Naldi, and Mohab Safey El~Din.
\newblock Exact algorithms for linear matrix inequalities.
\newblock {\em SIAM J. Optim.}, 26(4):2512--2539, 2016.

\bibitem{MR3840381}
Didier Henrion, Simone Naldi, and Mohab Safey El~Din.
\newblock Exact algorithms for semidefinite programs with degenerate feasible
  set.
\newblock In {\em I{SSAC}'18---{P}roceedings of the 2018 {ACM} {I}nternational
  {S}ymposium on {S}ymbolic and {A}lgebraic {C}omputation}, pages 191--198.
  ACM, New York, 2018.

\bibitem{kabatiansky78}
Grigorii~A. Kabatiansky and Vladimir~I. Levenshtein.
\newblock Bounds for packings on the sphere and in space.
\newblock {\em Problemy Peredachi Informacii}, 14(1):3--25, 1978.

\bibitem{MR2500392}
Erich Kaltofen, Bin Li, Zhengfeng Yang, and Lihong Zhi.
\newblock Exact certification of global optimality of approximate
  factorizations via rationalizing sums-of-squares with floating point scalars.
\newblock In {\em I{SSAC} 2008}, pages 155--163. ACM, New York, 2008.

\bibitem{MR2854844}
Erich~L. Kaltofen, Bin Li, Zhengfeng Yang, and Lihong Zhi.
\newblock Exact certification in global polynomial optimization via
  sums-of-squares of rational functions with rational coefficients.
\newblock {\em J. Symbolic Comput.}, 47(1):1--15, 2012.

\bibitem{kannan1979}
Ravindran Kannan and Achim Bachem.
\newblock Polynomial algorithms for computing the {S}mith and {H}ermite normal
  forms of an integer matrix.
\newblock {\em SIAM J. Comput.}, 8(4):499--507, 1979.

\bibitem{MR4044453}
Santiago Laplagne.
\newblock Facial reduction for exact polynomial sum of squares decomposition.
\newblock {\em Math. Comp.}, 89(322):859--877, 2020.

\bibitem{MR682664}
Arjen~K. Lenstra, Hendrik~W. Lenstra, Jr., and L\'{a}szl\'{o} Lov\'{a}sz.
\newblock Factoring polynomials with rational coefficients.
\newblock {\em Math. Ann.}, 261(4):515--534, 1982.

\bibitem{levenshtein79}
Vladimir~I. Levenshtein.
\newblock Boundaries for packings in {$n$}-dimensional {E}uclidean space.
\newblock {\em Dokl. Akad. Nauk SSSR}, 245(6):1299--1303, 1979.

\bibitem{machado16}
Fabr\'{\i}cio~Caluza Machado and Fernando~M. de~Oliveira~Filho.
\newblock Improving the semidefinite programming bound for the kissing number
  by exploiting polynomial symmetry.
\newblock {\em Exp. Math.}, 27(3):362--369, 2018.

\bibitem{magron2018exact}
Victor Magron and Mohab Safey~El Din.
\newblock On {E}xact {R}eznick, {H}ilbert-{A}rtin and {P}utinar's
  {R}epresentations.
\newblock {\em Preprint: arXiv:1811.10062}, 2018.

\bibitem{Hallah2013}
Rym M'Hallah, Abdulaziz Alkandari, and Nenad Mladenovi\'{c}.
\newblock Packing unit spheres into the smallest sphere using vns and nlp.
\newblock {\em Comput. Oper. Res.}, 40(2):603--615, February 2013.

\bibitem{MR3707858}
Simone Naldi.
\newblock Solving rank-constrained semidefinite programs in exact arithmetic.
\newblock {\em J. Symbolic Comput.}, 85:206--223, 2018.

\bibitem{MR2546336}
Jiawang Nie, Kristian Ranestad, and Bernd Sturmfels.
\newblock The algebraic degree of semidefinite programming.
\newblock {\em Math. Program.}, 122(2, Ser. A):379--405, 2010.

\bibitem{Odlyzko79}
Andrew~M. Odlyzko and Neil J.~A. Sloane.
\newblock New bounds on the number of unit spheres that can touch a unit sphere
  in {$n$} dimensions.
\newblock {\em J. Combin. Theory Ser. A}, 26(2):210--214, 1979.

\bibitem{MR3844533}
Frank Permenter and Pablo Parrilo.
\newblock Partial facial reduction: simplified, equivalent {SDP}s via
  approximations of the {PSD} cone.
\newblock {\em Math. Program.}, 171(1-2, Ser. A):1--54, 2018.

\bibitem{MR2474341}
Helfried Peyrl and Pablo~A. Parrilo.
\newblock Computing sum of squares decompositions with rational coefficients.
\newblock {\em Theoret. Comput. Sci.}, 409(2):269--281, 2008.

\bibitem{putinar93}
Mihai Putinar.
\newblock Positive polynomials on compact semi-algebraic sets.
\newblock {\em Indiana Univ. Math. J.}, 42(3):969--984, 1993.

\bibitem{schoenberg42}
Isaac~J. Schoenberg.
\newblock Positive definite functions on spheres.
\newblock {\em Duke Math. J.}, 9:96--108, 1942.

\bibitem{Schuette51}
Kurt Sch\"{u}tte and Bartel~L. van~der Waerden.
\newblock Auf welcher {K}ugel haben {$5$}, {$6$}, {$7$}, {$8$} oder {$9$}
  {P}unkte mit {M}indestabstand {E}ins {P}latz?
\newblock {\em Math. Ann.}, 123:96--124, 1951.

\bibitem{sloane81}
Neil J.~A. Sloane.
\newblock Tables of sphere packings and spherical codes.
\newblock {\em IEEE Trans. Inform. Theory}, 27(3):327--338, 1981.

\end{thebibliography}
\end{document}